\documentclass[12pt]{amsart}
\usepackage{amssymb,amsmath,amsthm}
\usepackage{tikz,enumerate}

\newcommand{\old}[1]{}
\theoremstyle{plain}
\newtheorem{thm}{Theorem}[section]
\newtheorem{lem}[thm]{Lemma}

\newtheorem{cor}[thm]{Corollary}

\theoremstyle{definition}
\newtheorem{defn}[thm]{Definition}

\newtheorem{ex}[thm]{Example}
\newtheorem{rk}[thm]{Remark}
\newtheorem{qn}[thm]{Question}

\def\reals{{\mathbb R}}

\title[$SB$-labelings and  poset topology]{$SB$-labelings and  posets with each interval homotopy equivalent to a sphere or a ball}

\author{Patricia Hersh} 
\address{Department of Mathematics, North Carolina State University, Raleigh, NC 27695}
\email{plhersh@ncsu.edu}

\author{Karola M\'esz\'aros}\address{Department of Mathematics, Cornell University, Ithaca, NY 14853}
\email{karola@math.cornell.edu}

\thanks{The authors thank the Banff International Research Station for providing a stimulating
environment in which they began to collaborate,  and specifically acknowledge support from NSF 
conference grant DMS-1101740.  The first  author was partially 
supported by NSF grants DMS-1002636, DMS-1200730 and DMS-1500987, and the second author was partially supported by an NSF Postdoctoral Research Fellowship, DMS-1103933, and NSF grant DMS-1501059.}


\begin{document}

\begin{abstract}
We introduce a new class of edge labelings  for locally finite lattices which we call  $SB$-labelings.  We 
 prove for finite lattices which admit  an $SB$-labeling that each open interval has the homotopy type of  a ball or  of a sphere of some dimension.
  Natural examples include  the weak order,  
  the Tamari lattice, and the finite distributive lattices.\\
 
\end{abstract}

\maketitle

\section{Introduction}

Anders Bj\"orner and Curtis Greene have raised the following question (personal communication of 
Bj\"orner;  see also \cite{Gr} by Greene).

\begin{qn}
Why are there so many posets with the property that every interval has M\"obius function equaling $0, 1$ or $-1$?  Is there a unifying explanation?
\end{qn}

This paper introduces  a new type of edge labeling that  a finite lattice may have  which we dub an 
$SB$-labeling.  We prove for finite lattices admitting such a labeling that each  open interval has order complex that  is contractible or is  homotopy equivalent to a sphere of some dimension.  This immediately yields that the M\"obius function only takes the values $0,\pm 1$ on all intervals of the lattice.  The construction and verification of validity of such labelings seems quite readily achievable on a variety of examples of interest.   The name  $SB$-labeling was chosen with $S$ and $B$  reflecting the possibility of spheres and balls, respectively.  
This method will easily  yield  that each interval in the weak Bruhat  order of a finite Coxeter group, in  the Tamari lattice, and in any finite distributive lattice  is  homotopy equivalent to a ball or a sphere of some dimension.    In particular, this method may be applied to non-shellable examples, as the weak Bruhat order for finite Coxeter groups will demonstrate.  Example ~\ref{dominance-example} will show  that not all finite lattices with M\"obius function taking only values $0,\pm 1$ admit an $SB$-labeling.   However, $SB$-labeling seems to be a convenient method for  studying homotopy type by capturing algebraic structure amongst cover relations and may be useful  for further applications to lattices which are  endowed with such structure.  

A motivation for the notion of an $SB$-labeling came from crystal graphs, an important  tool for studying  the representation theory of Kac--Moody algebras.  
The crystal graphs coming from highest weight  representations of Lie algebras in finite type (and more generally in symmetrizable Kac--Moody type) 
 are in fact Hasse diagrams of partially ordered sets.  They 
are naturally endowed with an edge labeling that by definition meets nearly all of  the requirements to be an $SB$-labeling (using  the upcoming formulation of an $SB$-labeling known as the index 2 formulation). 
The notion  of an $SB$-labeling has been used  in \cite{HL} to prove that any pair of elements $u,v$  satisfying $u<v$ with  $\mu (u,v)\not \in \{ 0,1,-1\} $ 
 in a finite crystal poset given by 
a highest weight 
representation in the simply laced case  also has the property  that the poset  interval $[u,v]$ contains within it a relation amongst the so-called crystal operators that is not implied by Stembridge's  local relations from \cite{Stembridge}.
This allowed Hersh and Lenart in \cite{HL} to discover arbitrarily high degree relations amongst crystal operators that are not implied by lower degree relations through  a computer search for intervals with M\"obius function taking values other than $0, 1, $ and $-1$.

It would be interesting to know of additional examples of finite (or locally finite)  lattices with $SB$-labelings.  We have not made a comprehensive search for such examples, but rather have chosen to focus in this paper on some well-known  families of lattices with the appropriate M\"obius function  that seemed  to us to be especially interesting classes of posets.  See \cite{Muhle} by Henri M\"uhle for further examples of posets with $SB$-labelings e.g.~ for posets derived from sortable elements in Coxeter groups.

Section ~\ref{disc-morse-bg-section} quickly reviews background that will be needed later in the paper.   Section ~\ref{labeling-section}
gives two different formulations for  the definition of an $SB$-labeling, and it is shown here  that the first of these two versions  of the definition for an  $SB$-labeling implies each open interval $(u,w)$ in a finite lattice $L$   is homotopy equivalent to a ball or a sphere, with the homotopy type being that of a sphere if and only if $w$ is a join of atoms of the interval.   Section ~\ref{equivalence-section}  proves that these two formulations of the definition for an $SB$-labeling  are equivalent to each other.  The value in this comes from the fact that the second formulation is a local condition that appears to be more easily verifiable for families of lattices of interest.    Section ~\ref{application-section} gives  applications: it provides   $SB$-labelings for  the finite distributive lattices, the weak order of any finite Coxeter group, and the Tamari lattice.

\section{Background} \label{disc-morse-bg-section}

   A  partially ordered set  (poset) $P$ is a {\bf lattice} if each pair of elements $x,y\in P$ has a unique least upper bound, which we denote $x\vee y$, and a unique greatest lower bound, which we denote $x\wedge y$.  We denote by $\hat{0}$ (resp.~  $\hat{1}$)  the unique minimal (resp.~ maximal)  element of a finite lattice.   A {\bf cover relation} $u\prec v  $ in a poset  $P$ is a pair of elements  $u <  v$ with the further requirement that $u\le z \le v$ implies either $u=z$ or $z=v$.   A chain $u_1 < u_2 < \cdots < u_k $ is said to be a {\bf saturated chain} from $u_1$ to $u_k$ if  each $u_i < u_{i+1}$ is a cover relation; a saturated chain is a {\bf maximal chain} of the poset  if additionally there do not exist any elements  $x$ satisfying either of the conditions  $x<u_1$ or $u_k < x$, i.e., if the chain is not properly contained in any other chains. 
   A maximal chain $u_1 < \cdots < u_k$ need not be the longest chain from $u_1$ to $u_k$ in terms of number of elements in the chain, and for non-graded lattices such as the Tamari lattice not all maximal chains are of maximal length.  
  An {\bf open interval } in $P$, denoted $(u,v)$,  is the subposet of elements $z\in P$ satisfying $u < z < v$.  Likewise, a {\bf closed interval} $[u,v]$ is the subposet comprised of those $z\in P$ such that $u\le z \le v$.  We will  sometimes refer to the open interval $(\hat{0},\hat{1})$ in a finite lattice $L$ as the {\bf proper part} of $L$.   For convenience, we pass freely back in forth  between speaking of the  saturated chains from $u$ to $v$, and the maximal chains of the closed interval  $[u,v]$.

The {\bf  M\"obius function}, denoted $\mu_P $,   of a finite partially ordered set  $P$ is defined recursively as follows.  
For each $u\in P$ we have $\mu_P (u,u) = 1$.  For each $u <  v$, $\mu_P(u,v) = -\sum_{u\le x < v} \mu_P(u,x)$.   The M\"obius function provides the coefficients in inclusion-exclusion counting formulas. 
The {\bf order complex} of a finite poset $P$  
is the simplicial complex, denoted 
$\Delta (P)$, whose $i$-faces
are chains 
$ v_0 < \cdots < v_i  $
of $i+1$ comparable poset 
elements.  It is well known for each $u<v$ in $P$ 
that $\mu_P (u,v)  = \tilde{\chi }(\Delta (u,v))  $ where $\Delta (u,v)$ denotes the order 
complex of the open interval $(u,v)$ and $\tilde{ \chi } $ denotes its reduced 
Euler characteristic (which is obtained from the usual Euler characteristic by subtracting one from it).
Sometimes we will speak of the homotopy type of a  poset or poset interval, by which we  mean the homotopy type of the order complex of that poset or poset interval.

Our focus throughout this paper  will be on posets
in which the order complex of each  open interval $(u,v)$ will turn out to  
be  homotopy equivalent to a ball or a 
sphere, implying that  $\tilde{\chi } (\Delta (u,v))$ and hence $\mu_P(u,v)$ equals $0, 1, $ or $-1$ 
for each pair $u<v$.  
A key tool underlying  our work will be 
the Crosscut Theorem, which we review next.

Recall  from \cite{bjorner-top-methods} (see also \cite{Bj-81}, \cite{Folkman}, \cite{Rota})
 that a subset $C$ of a poset $P$ is called a {\bf crosscut} if the following conditions hold.
\begin{enumerate}
\item
$C$ is an antichain.  
\item
For every finite chain $\sigma $ in $P$ there exists an element of $C$ that is comparable to every element of $\sigma $.
\item
For each $A\subseteq C$ which is bounded, i.e.,  which has an upper bound or a lower bound, then the  join or the meet of the elements of $A$ exists as an element of $P$.
\end{enumerate}

Define the {\bf crosscut complex} given by a crosscut $C$ to be the simplicial complex whose faces are those subsets of $C$ which are bounded.

\begin{rk}
In a finite lattice $L$  (and hence  also in the proper part of $L$), the set of atoms of $L$ is a crosscut.  In this paper, we will make use of the next theorem with the atoms as the chosen crosscut.
\end{rk}

\begin{thm}[Crosscut Theorem, Theorem 10.8 in \cite{bjorner-top-methods}]\label{crosscut-theorem}
The crosscut complex given by any crosscut of a finite poset $P$ is homotopy equivalent to the order complex of $P$.
\end{thm}

\begin{ex}
Letting $B_n$ denote the poset of subsets of $\{ 1,2,\dots ,n\} $ ordered by containment, notice that the crosscut complex for $B_3$ given by the crosscut comprised of the atoms of $B_3$ is a $2$-simplex.  On the other hand, consider the subposet $B_3\setminus \{ \hat{0}, \hat{1}\} $ in which we delete the unique minimal and maximal elements, namely we delete
 $\hat{0} = \emptyset $ and  $\hat{1} = \{ 1,2,3\} $; this smaller  poset 
has the same crosscut, but  its  crosscut complex is the boundary of a $2$-simplex.
\end{ex}

\begin{rk}\label{distinct-atom-crosscut}
If one can prove that distinct sets of atoms have distinct joins, then the Crosscut 
Theorem will imply that the subposet of joins of atoms has order complex that is 
homotopy equivalent to the order complex for the 
entire poset.  We will use this in the special case where our poset is the proper part of a finite lattice.
 An $SB$-labeling, a new type of edge labeling which we introduce momentarily, will guarantee that distinct sets of atoms have distinct joins.  
\end{rk}

\section{A new class of edge  labelings: $SB$-labelings}\label{labeling-section}

 Next we introduce a new class  of edge labelings  which we call 
 $SB$-labelings.  We will call a lattice admitting such a labeling an $SB$-lattice.   We will give two different formulations of the definition of an $SB$-labeling, and then  
 we will prove that these are equivalent to each other.   One formulation will be convenient for proving topological consequences of having an $SB$-labeling.  In particular, we use this formulation  to   prove that each open interval in a  finite lattice with an $SB$-labeling is homotopy equivalent to a ball or a sphere.  The other formulation seems likely to be more convenient for constructing $SB$-labelings on examples.  
 
 Later in the paper we will indeed demonstrate  that several well-known lattices admit $SB$-labelings, in spite of the fact that some of these lattices cannot possibly be shellable.
 Specifically, we  will apply this method  of $SB$-labeling 
 to the weak Bruhat order of a finite Coxeter group, 
 to the Tamari lattice,   and to the  finite 
 distributive lattices,   while  Example ~\ref{dominance-example} will show that dominance order on integer partitions does not in general admit an $SB$-labeling. 
 The $SB$-labelings will yield the homotopy type of each poset interval 
 by  a short, uniform approach for these classes of posets which had previously been analyzed by other methods.

  \begin{rk}
  It is natural to ask if  this notion for edge labelings may be extended to a more general notion for chain labelings (in the sense of  \cite{BW-on-lex}).  However, key properties of these $SB$-labelings in fact will  rely in an essential way on our usage of edge labelings rather than chain labelings.  Therefore, we confine ourselves to  considering edge labelings.
  \end{rk}

\begin{defn}\label{lattice-labeling}
An edge labeling $\lambda $ of a finite lattice $L$ is a {\bf lower $SB$-labeling} if it may be constructed as follows.  Begin with a label set $S$ such that there is a subset  $\{ \lambda_a | a\in A(L) \}  $ of $S$  whose members are in bijection with the set $A(L)$ of atoms of  $L$.
\begin{enumerate}
\item
No two labels upward from $\hat{0}$ to  distinct atoms may  be equal. 
This allows us to define  the label $\lambda_a $  on each   cover relation $\hat{0}\prec a$ as  the label corresponding to the  atom 
$a$.  
\item
Given any interval of the form 
$[\hat{0}, a_{i_1}\vee \cdots \vee a_{i_r} ] $ for  $ \{  a_{i_1},\dots
,a_{i_r} \}  \subseteq A(L)$, each of the saturated chains $M$ from $\hat{0} $ to $a_{i_1}\vee \cdots \vee
a_{i_r} $  has the property
that the set $\lambda (M)$ of labels occurring with positive multiplicity on $M$  is exactly 
$\{ \lambda_{a_{i_j}} | 1\le j \le r \} $.
\end{enumerate}
When an edge labeling $\lambda $ for a finite lattice $L$  meets  these conditions upon restriction to each  closed interval of $L$, then we call such a labeling an {\bf $SB$-labeling}.  We call a lattice  with an $SB$-labeling an {\bf $SB$-lattice}. 
\end{defn}

\begin{rk}
Notice that condition (2) above implies for $S, T$ distinct sets of atoms, that the join of the set of atoms in $S$ does not equal the join of the set of atoms in $T$.  In particular, this implies that the subposet of joins of atoms is a Boolean algebra.
\end{rk}

Now we give what we call  the  ``index 2 formulation of an  $SB$-labeling'', a type of labeling that we will prove in Theorem ~\ref{equivalent-definitions}  is equivalent to the notion of $SB$-labeling.   In light of Theorem ~\ref{equivalent-definitions}, one may henceforth take either definition as a definition of $SB$-labeling.

\begin{defn}\label{index-2-definition}
 The {\bf index 2 formulation of an  $SB$-labeling}  is an edge labeling on a  finite lattice $L$ satisfying the following conditions for each  $u,v,w \in L$ such that 
 $v$ and $w$ are distinct elements which each  cover $u$:
\begin{enumerate}[(i)]
\item
$\lambda (u,v)\ne \lambda (u,w)$
\item
Each saturated chain from $u$ to $v\vee w$  uses both of these labels $\lambda (u,v)$ and $\lambda (u,w)$ a positive number of times.
\item
None of the saturated chains from $u$ to $v\vee w$ use any other types of labels besides $\lambda (u,v)$ and $\lambda (u,w)$.
\end{enumerate}
\end{defn}

\begin{thm}\label{equivalent-definitions}
An edge labeling on a finite lattice is an $SB$-labeling if and only if it satisfies the index 2 formulation for an $SB$-labeling.
\end{thm}

\begin{proof}
Theorem ~\ref{cover-enough} proves that the  index 2 formulation of an  $SB$-labeling will always give an $SB$-labeling. 
On the other hand, if $\lambda $ is an $SB$-labeling, then Condition (1) for $SB$-labelings  directly gives Condition (i) in the index 2 formulation for an $SB$-labelings.    Condition (2)  for $SB$-labelings specialized to the  case of a join of two atoms yields exactly   conditions (ii) and (iii) of the index 2 formulation of an  $SB$-labeling.  
\end{proof}

\begin{ex} \label{ex:weak}
In the case of the weak Bruhat order of a finite Coxeter group, we will label each cover relation $u\prec s_i u$ with the label $s_i$ and will prove that this labeling  meets the requirements of the index 2 formulation of an $SB$-labeling.  Theorem ~\ref{weak-theorem} will verify that this is indeed an $SB$-labeling.  See Figure ~\ref{fig:weak} for an example.  To illustrate  part of the subtlety in this definition, notice e.g.~ that  the weak order interval $[\hat{0},s_1s_2]$ has a single maximal chain, and it uses the edge labels $s_1$ and $s_2$.  The label $s_2$  corresponds to an atom while the label $s_1$ does not.    In this weak order interval with only one atom, the conditions above hold vacuously since there is no triple of elements $u,v,w$ as in Definition ~\ref{index-2-definition}.
\end{ex}
\begin{figure}\label{fig:weak}

\begin{center}
\includegraphics[width=4cm]{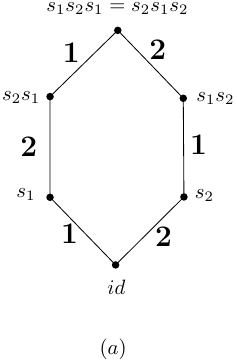}
\quad \quad \quad\quad
\includegraphics[width=5cm]{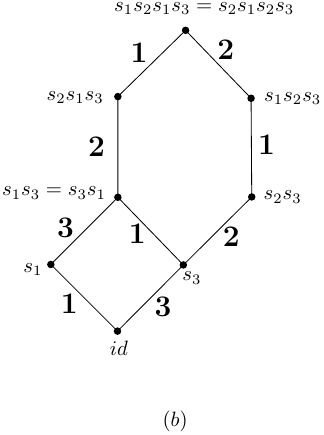}
\caption{ (a) The $SB$-labeling of the weak Bruhat order of $S_3$ as described in Example \ref{ex:weak}. The complex $\Delta (id,s_1s_2s_1)$ is homotopy equivalent to a sphere..
 \newline
(b)  The $SB$-labeling of an interval of the weak Bruhat order of $S_4$ as described in Example \ref{ex:weak}. The complex $\Delta (id,s_1s_2s_1s_3)$ is homotopy equivalent to a ball.}
\end{center}
\end{figure}
 
Next we show how the property of being an SB-lattice gives topological information about the order complex of each open interval in the lattice.

\begin{thm} \label{thm:sb}
If $L$ is an SB-lattice, then  each open interval is homotopy equivalent to a ball or a sphere of some dimension.  Moreover,  $\Delta (u,v)$ is homotopy equivalent to a sphere if and only if $v$ is a join of atoms of the interval, in which case it is a sphere $S^{d-2}$ where $d$ is the number of atoms in the interval.
\end{thm}

\begin{proof}
First note that it suffices to prove the result for lower intervals $(\hat{0},v)$, by virtue of how Definition ~\ref{lattice-labeling} is formulated.
Now each subset of the atoms  has a distinct join, due to the fact that  the set of labels appearing on the edges of all of the  saturated chains upward  from $\hat{0} $ to a join of atoms is exactly that set of atom labels.  
But this implies (cf.~ Remark ~\ref{distinct-atom-crosscut})   that the crosscut complex for $(\hat{0},v)$ given by the atoms is the boundary of a simplex if $v$ is a join of atoms and is the entire simplex otherwise.  In particular, this means that  the crosscut complex is homotopy equivalent to a sphere $S^{d-2}$ if $v$ is a join of atoms and is contractible otherwise.
Now the Crosscut Theorem (which we recall as Theorem ~\ref{crosscut-theorem}) yields the result.  
\end{proof}

We conclude this section with some relaxations that may be made in the hypotheses of our main results without changing the conclusions.

\begin{rk}\label{locally-finite-variant}
In the notion of an $SB$-labeling, we may replace the finiteness requirement for our lattices by instead requiring them to be locally  finite  with  a unique minimal element.   
Our proofs all go through unchanged in such cases, allowing us to call such lattices $SB$-lattices and draw all of the same conclusions.  Young's lattice will provide one such example, as we will show in Section ~\ref{application-section}.  
\end{rk}

\begin{defn}
Let us say that a finite poset $P$ with unique minimal and maximal elements  is an {\bf atom-near-lattice} if each pair of elements $u,v \in P$ with $u<v$  has the property that each collection $S$  of atoms of the closed interval $[u,v]$ has a unique least upper bound $\vee_{a\in S} a$. 
\end{defn}

\begin{rk}
It is proven in Lemma 2.1 of  \cite{BEZ} that this atom-near-lattice property in fact implies that $P$ is a lattice.  
This property may be easier to check in examples of interest than the property of being a lattice.  Our proofs actually only rely upon this formulation of the lattice property.
\end{rk}

\section{Index 2 formulation is equivalent to an $SB$-labeling}\label{equivalence-section}

This section proves the equivalence of our two different definitions for an $SB$-labeling.
To this end, we will use the next  two  notions to prove that every labeling meeting the conditions in the index 2 formulation for an $SB$-labeling is an $SB$-labeling.

\begin{defn}
We say that a pair of maximal chains $M_1, N_1$ in a finite lattice
are connected by a {\bf basic move} if $M_1$ and $N_1$  coincide except on an open interval $(u,v)$ where $u\prec x$ in $M_1$ and $u\prec y$ in $N_1$ with $x\ne y$ and with  $v = x\vee y$.  
\end{defn}

\begin{ex}\label{braid-special}
For example, in any interval $[u,wu]$ in the weak order the basic moves are given by the long and short braid moves on reduced expressions for the  Coxeter group element $w$. 
\end{ex}

\begin{defn}
Define the {\bf total length} of a closed interval $[u,v]$ to be the sum of the lengths of all the saturated chains in that interval.
\end{defn}
 
 This notion of total length is designed to enable a proof by induction in the next lemma without needing the lattices to be graded and without needing to require the maximal chains of $[a,b]$   to be of maximal length amongst all maximal chains on $[a,b]$.

\begin{lem} \label{sat}  
Any two maximal chains on an interval $[a,b]$ in a finite lattice $L$ are  connected by a series of basic moves.
\end{lem}

\proof 
Suppose otherwise.    Then choose   an interval $[u,v]$ where this fails, making   the total length of the interval  as small as possible among all such examples.  Let $M_1$ and $N_1$ be two maximal chains on $[u,v]$ that are not connected by a series of basic moves.  Our minimality assumption on total length  ensures for $u\prec u_1$ in $M_1$ and $u\prec v_1$ in $N_1$ that we must have $u_1\ne v_1$.  We also may assume   $u_1\vee v_1 \ne v$, since otherwise there would be a single basic move connecting $M_1$ to $N_1$ by definition of basic move.

 Our plan in this case is to give a series of  steps $M_1\rightarrow M_2\rightarrow M_3\rightarrow N_1$ which convert $M_1$ to $N_1$ and to show that each of these three  steps may  be achieved through a series of basic moves, hence that their composition may  as well. 
 Let $M_2$ be a maximal chain on $[u, v]$ which  agrees with $M_1$ except possibly on $(u_1,v)$;  $M_2$ is chosen to include $u_1\vee v_1$ (since $v\ne u_1\vee v_1$, we have $u_1\vee v_1 < v$). Since the  interval $[u_1, v]$ has strictly smaller total length than $[u,v]$ and $M_1$ agrees with $M_2$ except on  this interval, we can conclude  there is  a  series of  basic moves converting the restriction of $M_1$ to $[u_1,v ]$ to the restriction of $M_2$ to this same interval, which in turn gives basic moves converting $M_1$ to $M_2$ in $[u,v]$.  Now we similarly may convert $M_2$ to a maximal chain $M_3$ which coincides with $M_2$ except on the interval $(u,u_1\vee v_1)$ and which has $u_1$ replaced by $v_1$;  this interval  also has strictly smaller total length than $[u,v]$, again implying the desired basic moves.  Finally, we note that $M_3$ only differs from $N_1$ on the proper part of the interval $[v_1,v]$, which yet again has strictly smaller total length than $[u,v]$, enabling us to find a series of basic moves converting $M_3$ to $N_1$, completing the result.
\qed

\begin{rk}
Lemma ~\ref{sat}  may be regarded as an abstraction of the idea that the lattice property for the weak Bruhat order of a finite Coxeter group  ensures that any two reduced expressions for the same Coxeter group element are connected by a series of long and short braid moves. 
 This implication in the case of the weak order appears as Theorem 3.3.1 in \cite{BB}.
\end{rk}

\begin{thm}\label{cover-enough}
If a finite lattice $L$ has an edge labeling that  satisfies the index 2 formulation for  an $SB$-labeling, then it is an $SB$-labeling.
\end{thm}

\begin{proof}
Let $\lambda $ be an edge labeling for a finite lattice $L$ which meets the requirements for the index 2 formulation of an $SB$-labeling.  We will prove by induction on the number $r$ of atoms  that $\lambda $ also meets the requirements to be a lower $SB$-labeling.  In fact, this will imply $\lambda $ is an $SB$-labeling, since  applying this argument to any closed interval will  show we have a lower $SB$-labeling for each closed interval.

  The base case with 1 atom is tautologically true.  
  Let us suppose that   $\{ a_{i_1},\dots ,a_{i_r} \} $ is the set of atoms of $L$.  
 Now consider the interval  $L_{r-1} = [\hat{0},a_{i_1}\vee \cdots \vee a_{i_{r-1}}]$  within $L$.  By  induction, we may assume that  this  uses only the labels $\{ a_{i_1},\dots ,a_{i_{r-1}} \} $.  We will progressively build from $L_{r-1}$ a larger subposet $L_{r-1,1}$  of $L$ all of whose cover relations are cover relations of $L$ with the further property that it includes an upper bound $m$ for $\{ a_{i_1}\dots ,a_{i_r} \} $.  We will deduce from $a_{i_1}\vee \cdots \vee a_{i_r} \le m$ that $[\hat{0},a_{i_1}\vee\cdots\vee a_{i_r} ]$ also uses at most the labels $\{ a_{i_1},\dots ,a_{i_r}\} $.
 Finally, we will also  show that each saturated chain from $\hat{0}$ to 
$a_{i_1}\vee \cdots \vee a_{i_r}$ in fact uses all of these labels.

First we add to  $L_{r-1}$  the additional atom  $a_{i_r} $ as well as all elements belonging to  the closed interval $[\hat{0},  a_{i_1}\vee a_{i_r}]$   to obtain a new  poset  $L_{r-1}^{(1)}$.  By condition (iii) in the index 2 formulation of an $SB$-labeling, this slightly larger poset  still only uses the allowed edge labels, since all of the new cover relations are in the interval $[0,a_{i_1}\vee a_{i_r}]$ which only uses the labels $a_{i_1}$ and $a_{i_r}$.  Now either $a_{i_1}\vee a_{i_r}\in L_{r-1}$, in which case we are done constructing $L_{r-1,1}$, or there are at least two different maximal elements in
$L_{r-1}^{(1)}$.  For each maximal element $m_i^{(1)} \in L_{r-1}^{(1)}$, let $P_i$ be the subposet of elements  $x\in L_{r-1}^{(1)}$ satisfying $x\le m_i^{(1)}$.  Choose $u^{(1)}$ to be an element that is contained in both $P_j$ and $P_k$ for some  $j\ne k$ such that there are no elements strictly greater than $u^{(1)}$ also having the property of being contained in some $P_{j'}$ as well as some $P_{k'}$ for $j'\ne k'$; finiteness of $L_{r-1}^{(1)}$ guarantees the existence of such an element $u^{(1)}$.

 Now consider cover relations 
$u^{(1)}\prec x_1^{(1)}$ and $u^{(1)}\prec x_2^{(1)}$ in $L_{r-1}^{(1)}$  such that $x_1^{(1)} \le m_j^{(1)}$ and $x_2^{(1)} \le m_k^{(1)}$ in $L_{r-1}^{(1)}$.  Obtain from $L_{r-1}^{(1)}$ a strictly  larger poset $L_{r-1}^{(2)}$ by adding all elements and cover relations from the interval $[u^{(1)},x_1^{(1)} \vee x_2^{(1)} ]$.  Again by condition   (iii), this cannot introduce any new labels.
 Again, we either have a unique maximal element or we  have two different maximal elements, allowing us to apply this same procedure and do so  repeatedly until we have a unique maximal element $m$.  Specifically, at the $n$-th iteration of the procedure, the input is a poset  $L_{r-1}^{(n-1)}$ having distinct  maximal elements. 
  This allows us to find an element $u^{(n)}$ satisfying the same criterion at this step that  $u^{(1)}$ satisfied at the first step, now using distinct maximal elements $m_j^{(n-1)}$ and $m_k^{(n-1)}$ from
  $L_{r-1}^{(n-1)}$.  
  This in turn ensures there are elements $x_1^{(n)}$ and $x_2^{(n)}$ with $u^{(n)} \prec x_1^{(n)} \le m_j^{(n-1)}$ and $u^{(n)}\prec x_2^{(n)} \le m_k^{(n-1)}$ such that $x_1^{(n)}$ is not less than or equal to any maximal element other than $m_j^{(n-1)}$,  and  
  $x_2^{(n)}$ is not less than or equal to any maximal element other than $m_k^{(n-1)}$.
  We obtain $L_{r-1}^{(n)}$ as the poset $L_{r-1}^{(n-1)}$  with the additional elements and cover relations from the  interval $[u^{(n)},x_1^{(n)} \vee x_2^{(n)}]$  added to it.

 We  iterate  this process
   until it yields a poset $L_{r-1,1}$  with a unique maximal element $m$.
 This process must terminate within finitely many iterations due to finiteness of our original  lattice.
    By construction, the unique maximal element $m$  of $L_{r-1,1}$
         will  be an upper bound for $\{ a_{i_1},\dots ,a_{i_r} \} $, and we will  have only used the labels $a_{i_1},\dots ,a_{i_r}$ on the poset $L_{r-1,1}$  obtained by this process.   The fact that we only ever insert cover relations from the original lattice implies that each saturated chain in $L_{r-1,1}$  from $\hat{0}$ to $m$ is also a saturated chain in the original lattice $L$.    Since 
$a_{i_1}\vee\cdots \vee a_{i_r} \le m$ in $L$,   there is a saturated chain from $\hat{0} $ to $m$ in $L$ which includes the element $a_{i_1}\vee \cdots \vee a_{i_r}$.  By   Lemma ~\ref{same}, this  implies that  the set of labels on each saturated chain from $\hat{0}$  to $a_{i_1}\vee \cdots \vee a_{i_r}$  must be a subset of the set of labels on a saturated chain from $\hat{0}$ to $m$.  Thus,  no labels other than $a_{i_1},\dots ,a_{i_r}$ appear on any saturated chain  from $\hat{0} $ to $a_{i_1}\vee\cdots \vee a_{i_r}$. 

Now let us show  that each saturated chain from $\hat{0} $ to $a_{i_1}\vee \cdots \vee a_{i_r}$ uses each of the   labels $\{ a_{i_1}, \dots ,a_{i_r}\} $ a positive number of times.
 The point is that each atom $a_{i_j}$ for $1\le j\le r$  is in some maximal chain in $[\hat{0}, a_{i_1}\vee\cdots \vee a_{i_r}]$, implying that there exists a maximal chain using the label $a_{i_j}$;
   but this implies that all maximal chains use $a_{i_j}$ for each $1\le j\le r$, by Lemma ~\ref{same}. 
 In conclusion, we have shown that each maximal chain uses exactly the set of labels $\{ a_{i_1},\dots ,a_{i_r} \} $,  each with positive multiplicity.
\end{proof}

\begin{cor}
An edge labeling of a  finite lattice is an $SB$-labeling if and only if it satisfies the index 2 formulation of an
$SB$-labeling.
\end{cor}

\begin{proof}
One direction is proven as Theorem ~\ref{cover-enough}.  The other direction follows immediately  from the fact that the index 2 formulation of an $SB$-labeling is a special case of the notion of an $SB$-labeling.  
\end{proof}

To complete the proof of Theorem ~\ref{cover-enough}, 
what remains is to prove Lemma ~\ref{same}, using Lemma ~\ref{2a} below to do this.

\begin{lem} \label{2a} 
Let  $L$ be a finite lattice with an edge labeling $\lambda $  which  satisfies the index 2 formulation for an $SB$-labeling.  Then for any $u,v,w\in L $ with $v$ and $w$ both covering $u$, the  interval $[u,v\vee w]$   cannot have any atoms other than $v$ and $w$.
\end{lem}

\proof
If there were another atom $x$ in $[u,v\vee w]$, then we must have $\lambda (u,x) \ne \lambda (u,v)$ and $\lambda (u,x) \ne \lambda (u,w)$.   We also must have $x\vee v\le v\vee w$ since $x\le v\vee w$ and $v \le v\vee w$.  By virtue of $\lambda $ meeting the index 2 formulation for an  $SB$-labeling, we have that every maximal  chain on the interval $[u,x\vee v]$ must use the label $\lambda (u,x)$.  But then  there will be maximal  chains on the interval $[u,v\vee w]$ which
also must use the label $\lambda (u,x)$, a contradiction.
\qed
 
 \medskip

 \begin{lem} \label{same}
 If  an edge labeling $\lambda $ on a finite lattice $L$ 
 meets the conditions for the index 2 formulation of an $SB$-labeling, then this   guarantees for each interval $[a,b]$ in $L$  that any two saturated chains $M_1$ and $N_1$  on $[a,b]$  must   use the same set of labels  each a positive number of times, though not necessarily  with the same multiplicities. 
\end{lem}

\begin{proof} 
 Lemma \ref{2a} implies that   any two maximal  chains $M_1$ and $N_1$  on an interval $[a,b]$  in a finite lattice  that are connected by a series of basic moves  use the same set of labels (though not necessarily with the same multiplicities).  Lemma \ref{sat} checks that any two saturated chains $M_1$ and $N_1$  from $a$ to $b$ in a finite lattice  are connected by a series of basic moves, so  the result follows. 
 \end{proof}

\section{Applications}\label{application-section}

Now we turn to applications, beginning with finite distributive lattices.  In this first example, the $SB$-labeling  we give is  also a well-known  $EL$-labeling, implying the posets are shellable.  The homotopy type of the intervals in finite distributive lattices was determined in \cite{Bjorner}  indirectly by virtue of finite distributive lattices also  being finite  supersolvable lattices, relying on an earlier $R$-labeling given by Stanley in \cite{Stanley}  for finite supersolvable lattices.

\begin{thm}\label{finite-distributive-lattice}
Any finite distributive lattice  is an $SB$-lattice.   
\end{thm}

\begin{proof}
We  will use the fact that any finite distribute lattice  $L$ is the poset $J(P)$ of order ideals of a finite poset $P$ ordered by inclusion (cf.~ \cite{ec1}, Theorem 3.4.1).  This allows us to regard each cover relation $u\prec v$ as adding to the order ideal associated to $u$ 
a single element $p\in P$  We use this element $p$  as the label for $u\prec v$.  Whenever we have $u\prec v$ and $u\prec w$, then this implies that there are two different elements of $P$, either of which  may individually  be added to the order ideal given by $u$ to obtain a new order ideal.  Therefore, $v\vee w$ covers both $v$ and $w$ with the further property that there cannot be any other elements $z$ with $u < z < v\vee w$.  From this,  Conditions (i), (ii) and (iii)  for the index 2 formulation for an  $SB$-labeling follow directly.
\end{proof}

Recall that Young's lattice is the poset of integer partitions regarded as Young diagrams, with $u\prec v$ whenever $v$ is obtained from $u$ by adding a single box.  
Since Young's lattice is a locally finite,  distributive lattice with a unique minimal element,  Theorem ~\ref{finite-distributive-lattice} together with Remark ~\ref{locally-finite-variant} allows us to  conclude:
 
\begin{cor}
Young's lattice is an  $SB$-lattice. 
\end{cor}

Next we turn to a non-shellable example, the weak Bruhat order  on the elements of a  finite Coxeter group $W$.
Let $S$ be the set of  simple reflections generating $W$.   The  (left) 
weak Bruhat order has as its  cover relations each  $w\prec s_iw$ for $w\in W $ and $s_i\in S$ with $l(w) < l(s_i w)$, letting  $l(w)$ denote the Coxeter-theoretic length of $w$.  See e.g. \cite{BB} or \cite{Hu} for further background on Coxeter groups and on the  weak Bruhat order.    The homotopy type of each interval was originally determined in \cite{bj} (see also \cite{Ed} and \cite{EW} 
for related results regarding posets of regions).  We will use the following properties of weak Bruhat order.

\begin{enumerate}
\item
There is an 
isomorphism of weak Bruhat order
intervals $[u,w] \simeq [e,wu^{-1}]$ for each $u\le w$.  This appears as 
Proposition 3.1.6 in \cite{BB}.
\item  
 There is a characterization of the joins of finite sets of atoms in Lemma 3.2.3 in \cite{BB} as exactly those Coxeter group elements which are longest elements of parabolic subgroups of $W$; specifically, the join of the collection of atoms corresponding to the simple reflections in $S = \{ s_{i_1},\dots ,s_{i_r} \} $ will be  $w_\circ (S)$, namely the longest element of the parabolic subgroup $W_S$ generated by the elements of $S$.
  \item
 There is a bijection between the maximal chains in the interval $[u,w]$ and the reduced expressions for $wu^{-1}$. This appears in Proposition 3.1.2, part (i),  in \cite{BB}.
 \item
 The atoms of $[e,w_\circ (S)]$ are exactly the atoms $a_s$ for $s\in S$, namely the atoms naturally corresponding to the elements of $S$.  A proof of this may be found in  
 Propositions 3.1.2, parts (iii) and (iv)  in \cite{BB}.
 \item
 The weak Bruhat order is a lattice.  This is proven e.g.~ in Theorem 3.2.1 in \cite{BB}
\end{enumerate}

\begin{thm}\label{weak-theorem}
The weak Bruhat  order for any finite Coxeter group $W$ is an  $SB$-lattice.   Moreover, an open  interval $(u,w)$ in $W$  is homotopy equivalent to a sphere $S^{d-2}$ if  $wu^{-1}$ is the longest element of a parabolic subgroup $W_S$ with $|S|=d$, and $(u,w)$ is contractible otherwise.
\end{thm}

\begin{proof}
We will prove the result for left weak Bruhat order, noting that a completely analogous proof holds
for right weak Bruhat order.
By property (5) above, the weak Bruhat order of a finite Coxeter group  is a lattice.
To obtain an $SB$-labeling, we will label each cover relation
$u\prec  v$ with the unique  simple reflection $s_i$ such that $v = s_iu$.

Next let us  justify 
the following  claim: for any two distinct
cover relations $u\prec v$ and $u\prec w$, there are unique saturated chains $u\prec v\prec \cdots \prec v\vee w$ and $u\prec w\prec\cdots \prec v\vee w$ from  $u$ to $v\vee w$ and no other saturated chains from  $u$ to $v\vee w$; moreover,  these two saturated chains  have label sequences $s_i s_j s_i \cdots $ and $s_j s_i s_j \cdots $  each consisting of an alternation of only  the letters $s_i$ and $s_j$, with each label sequence having the same length $m(i,j)$ where $m(i,j)$ is the order of the Coxeter group element $s_is_j$.  In the case with $u=e$, this claim  holds  by property (2) above.  Otherwise, we use property (1) above to reduce the claim  to the case with $u=e$.    This claim shows that all of the requirements for the index 2 formulation of an $SB$-labeling are indeed satisfied, implying that weak Bruhat order is an $SB$-lattice.

Properties  (1) and (2) above also combine to imply that $w$ will be a join of atoms of the interval $[u,w]$  if and only if $wu^{-1}$ is the longest element  $w_\circ (W_S)$ 
of some parabolic subgroup  $W_S$ of $W$.  
Letting $A$ be the set of atoms of $[u,w]$, Theorem 
~\ref{thm:sb} implies  that  $[u,w]$ has the homotopy type of a sphere $S^{|A|-2}$ 
  if and only if $wu^{-1}$ is the longest element of some parabolic subgroup 
$W_S$ of $W$ and is contractible otherwise. 
In the case that $wu^{-1}  = w_\circ (W_S)$, 
property (4) above  
 guarantees that $ |A| = |S| $.  
\end{proof}

The Tamari lattice, our next example,  is a partial order on the 
 triangulations of an $(n+1)$-gon, as discussed e.g. in \cite{ER}.  To define the cover relations, it is 
 convenient to label the $n+1$ vertices of the $(n+1)$-gon with the integers $1,2,\dots ,n+1$, proceeding clockwise about the boundary of the $(n+1)$-gon.  Now each cover relation $u\prec v$ will replace an edge $e_{i,k}$ in a triangulation $u$ by an edge $e_{j,l}$ to obtain the triangulation $v$, subject to  the requirement  that we have  $i<j<k<l$. See Figure \ref{fig:tamari} for an illustration.

\begin{figure}
\centering
\includegraphics[width=10cm]{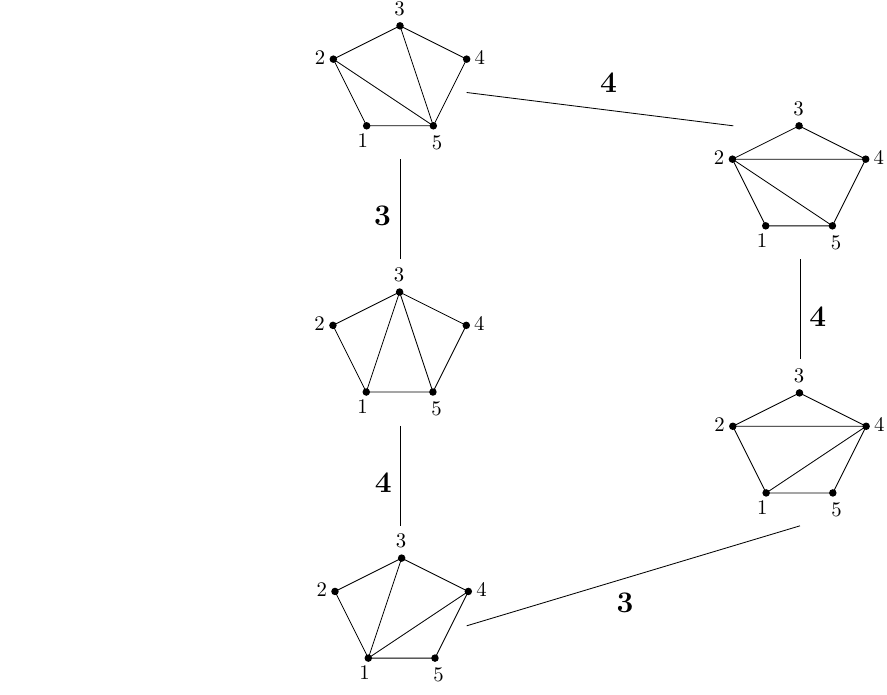}
\caption{The Tamari lattice for $n=4$ with the $SB$-labeling from the proof of Theorem \ref{thm:tamari}.} 
\label{fig:tamari}
\end{figure}

 The  significance of the Tamari lattice comes in part from the fact that its Hasse diagram is the 1-skeleton of the associahedron, a polytope which goes back to work on homotopy associative $H$-spaces by Stasheff (cf.~ \cite{Stasheff}).  The number of elements in the Tamari lattice is a Catalan number.  It was proven to be non-pure shellable with each interval having the homotopy type of a ball or a sphere by Bj\"orner and Wachs in \cite{BW}.  Earlier results regarding its M\"obius function and implicitly regarding its topological structure also appear in \cite{Pallo}.
 See e.g.~  \cite{BW}, \cite{HT}, \cite{Lo}  for further background on the Tamari lattice.  
  
\begin{rk}
The symmetry in this definition for the Tamari lattice guarantees that it will be self-dual, so that having an $SB$-labeling will be equivalent to its dual poset  having an $SB$-labeling.
\end{rk}

\begin{thm} \label{thm:tamari}
The Tamari lattice is an SB-lattice.
\end{thm}

\begin{proof} 
 For our proposed $SB$-labeling $\lambda $,  we record for each cover relation 
 $v\prec w$ the letter $\lambda(v,w) = k$ such 
 that an edge $e_{i,k}$ in $v$ is replaced by an edge 
 $e_{j,l}$ to obtain  $w$ for some $i<j<k<l$.  Lemma ~\ref{first-claim} proves that there is at most one cover
 relation upward from $v$ with this label $k$, thereby proving condition (i) in the index 2 formulation
 for an  $SB$-labeling.  Let us denote by 
 $u_k(v)$  the element $w$ with $v\prec w$ and $\lambda (v,w)=k$, 
  when such an element $w$ exists.
   
In order to prove conditions (ii) and (iii) in the index 2 formulation for an $SB$-labeling, it will suffice  to  confirm  the following facts regarding any element $v$ in the Tamari lattice and any pair of distinct elements $u_i(v)$ and $u_j(v)$ both covering $v$ with $i<j$.
\begin{enumerate}
\item
 We have  either $u_iu_j(v)=u_ju_i(v)$ or $u_iu_j(v)=u_ju_ju_i(v)$; moreover,  $u_i(v) \vee u_j (v) = u_iu_j(v)$ in either case.
 \item
 The only possible  elements in the interval $[v,u_i(v)\vee u_j(v)]$  
 are  $v, u_i(v), u_j(v), u_iu_j(v)$ and $u_ju_i(v)$.
  \end{enumerate}
  Both of these claims are proven in Lemma ~\ref{second-claim}.
  \end{proof}

\begin{lem}\label{first-claim}
Each element $v$ is covered by at most one element $w$ such that $\lambda (v,w) = k$ for any given label $k$.
\end{lem}

\begin{proof}
Given 
a triangulated  $(n+1)$-gon $v$, 
consider the edges emanating outward from some fixed  vertex $k$ for $2\le k\le n$, proceeding in counterclockwise order through this list of edges, and including in it the boundary edges 
$e_{k-1,k}$ and $e_{k,k+1}$ of the $(n+1)$-gon. That is, 
consider the maximal 
sequence  of edges having  the form $e_{j_r,k}, e_{j_{r-1},k},\dots ,e_{j_1,k}, e_{k,l_s},
e_{k,l_{s-1}},\dots ,e_{k,l_1}$ for $j_1<j_2<\cdots <j_r<k<l_1<l_2<\cdots <l_s$.  When there exists 
some $w$ such that $v\prec w$ with $\lambda (v,w)=k$, this implies $r\ge 2$ and $s\ge 1$.  
Notice that for a cover relation $v\prec w$ to have $\lambda (v,w) = k$, it must be replacing the
edge $e_{j_1,k}$ in $v$  by an edge $e_{j_2,l_s}$ to obtain $w$.  
 In particular, this means there is at most one cover relation $v\prec w$  colored $k$.  
\end{proof}

\begin{lem}\label{second-claim}
  Suppose $v$ is covered by distinct elements $u_i(v)$ and $u_j(v)$ for $i<j$.  If
  $u_iu_j(v) = u_ju_i(v)$, then the interval 
  $[v,u_i(v)\vee u_j(v)]$ has exactly 4 elements.
   If $u_iu_j(v)\ne u_ju_i(v)$,
  then $u_iu_j(v) = u_j u_j u_i(v)$, and  the interval $[v,u_i(v)\vee u_j(v)]$  has exactly 
  5 elements, namely the elements  
  $v$, $u_i(v)$, $u_j(v)$, $u_iu_j(v)$ and $u_ju_i(v)$. In either case, 
  $u_i(v)\vee u_j(v) = u_iu_j(v)$.
  \end{lem}

  \begin{proof}
   In the case where $v\prec u_i(v)$ replaces an edge across one quadrilateral of $v$ while   $v\prec u_j(v)$ replaces an edge across  another  
  quadrilateral of $v$ whose interior is 
  completely disjoint from the interior  of the first quadrilateral, these two edge replacement 
  operations commute, yielding $u_iu_j(v)=u_ju_i(v)$; by construction,  the interval from $v$ to $u_iu_j(v)$ will  then have exactly 4  elements.  To see that 
  $u_iu_j(v) = u_i(v) \vee u_j(v)$ in this case, 
  we use that the Tamari lattice is a lattice and 
  that there is not room for a strictly lower upper bound for $u_i(v)$ and $u_j(v)$  by virtue of the
 definition of cover relation.  
  
  Suppose on the other hand that  the pair of quadrilaterals 
  that are triangulated by   
  the two  edges to be flipped from their positions in $v$  by the two cover 
  relations $v\prec u_i(v)$ and $v\prec u_j(v)$ have interiors that 
   are not disjoint.  Then these triangulated 
  quadrilaterals must  share a triangle.  This forces the union of the two triangulated  quadrilaterals to comprise a triangulated pentagon appearing  within  the triangulation $v$. 
   In order to have cover relations  $v\prec u_i(v)$ and $v\prec u_j(v)$ both proceeding upward from 
   $v$, one may check directly that  the triangulated pentagon within $v$ 
   must have 5 vertices labeled $a,b,i,j,k$ for $a<b<i<j<k$ with edges
  $e_{a,i}$ and $e_{a,j}$ that can be flipped by applying the operators  $u_i$ and $u_j$, 
  respectively,   to give the new edges $e_{b,j} $ 
  and $e_{i,k}$, respectively.  
  One  may likewise check directly that $[v,u_iu_j(v)]$ has exactly the desired 5 elements.  By 
  virtue of the incomparability
   of $u_i(v)$ and $u_j(v)$ together with the definition of cover relation, there cannot be 
  any $z$ satisfying the three conditions $u_i(v)\le z$ and $u_j(v)\le z$ and $z < u_iu_j(v)$.  
  Therefore, $u_iu_j(v)$ is a least upper bound for $u_i(v)$ and $u_j(v)$.  
  Since the Tamari lattice is well known to be a 
  lattice, this must be the unique least upper bound $u_i(v)\vee u_j(v)$.
   \end{proof}

\begin{ex}\label{dominance-example}
Neither the dominance order on the partitions of an integer $n$ nor its dual poset admits an $SB$-labeling in general.  This can be seen by considering  the interval downward from the partition $(5,4,3,2,1)$ to the meet of the 4 elements that are covered by  $(5,4,3,2,1)$, namely the interval shown in Figure \ref{fig:dominance}; we leave it as an exercise for the reader to check that neither the poset shown in Figure \ref{fig:dominance} nor its dual poset admits an $SB$-labeling, implying the same for dominance order for $n=15$.
Note that the dominance order is an  $I$-lattice, a notion introduced by Greene in \cite{Gr} for proving that certain lattices take only M\"obius function values $0, 1, -1$. This example  shows that the notions of $I$-lattices and $SB$-lattices are distinct.

 Recall that the dominance order was proven to be non-pure shellable with  each open interval homotopy equivalent to a  ball or a sphere  in \cite{BW}.  The M\"obius function was  determined prior to the development of the notion of non-pure shellability in \cite{Bogart}, \cite{Brylawski} and \cite{Gr}.  
\end{ex}

\begin{figure}
\centering
\includegraphics[width=8cm]{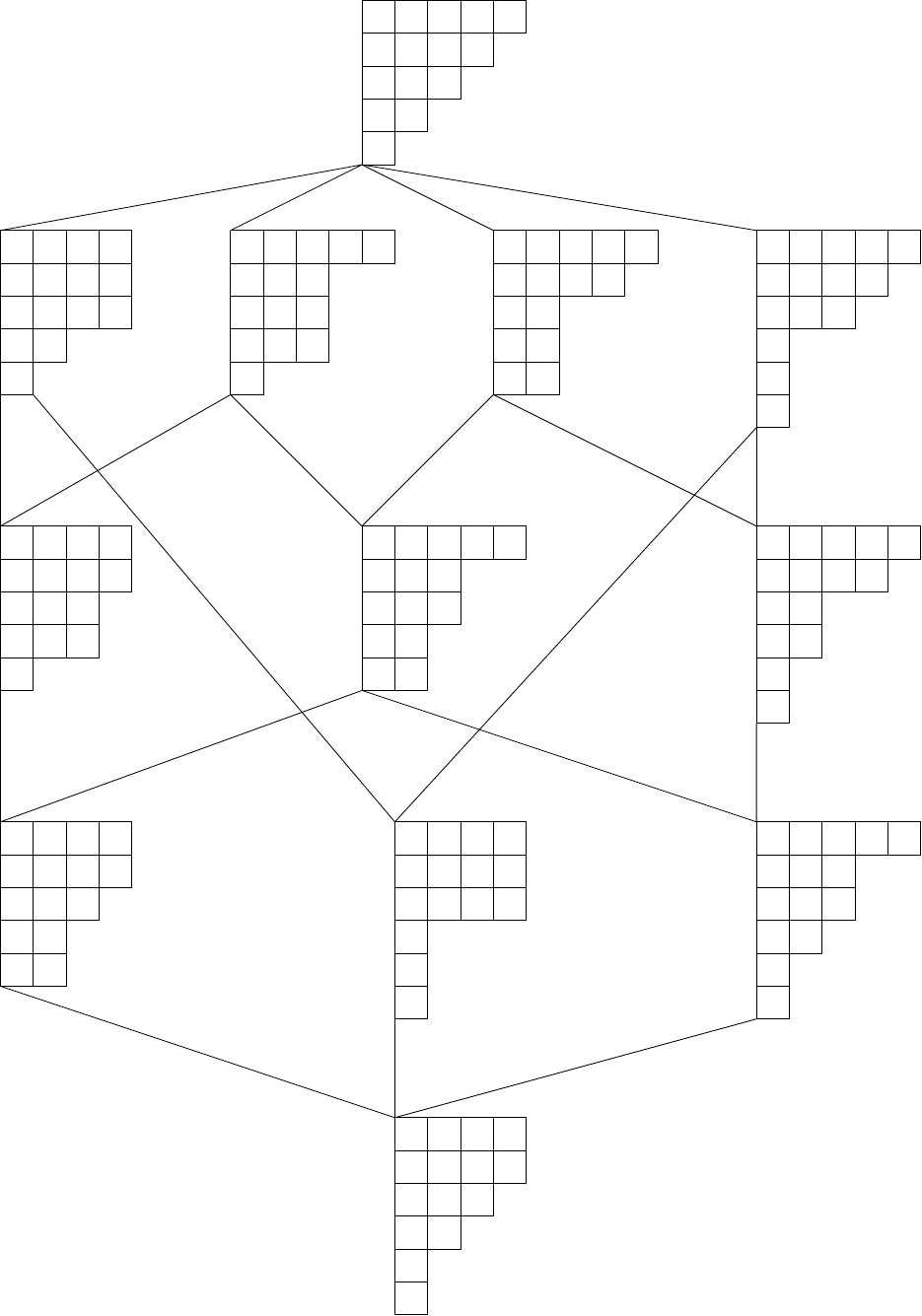}
\caption{The interval between partitions $(4,4,3,2,1,1)$ and $(5,4,3,2,1)$ in the dominance order. This interval, as well as  its dual, has no $SB$-labeling.} 
\label{fig:dominance}
\end{figure}

\section*{Acknowledgments}

The authors  are  grateful to Georgia Benkart, Stephanie van Willigenburg,  Monica Vazirani, and  the Banff International Research Station (BIRS) for conducting an inspiring  workshop entitled
Algebraic Combinatorixx  in May 2011 for female researchers in algebraic combinatorics with the goal of helping  establish new and fruitful collaborations.   This project grew out of discussions that began at that workshop.
 The authors also thank  Louis Billera, Anders Bj\"orner, Curtis Greene,  Thomas McConville,
  Peter McNamara,   Vic Reiner,  Hugh Thomas, and the very helpful anonymous referees
 for helpful discussions, references, and suggestions.

\end{document}